\documentclass[12pt,reqno]{amsart}
\usepackage{amssymb}

\let\Bbb\mathbb

\usepackage{eucal}
\let\euc\mathcal

\def\>{\relax\ifmmode\mskip.666667\thinmuskip\relax\else\kern.111111em\fi}
\def\<{\relax\ifmmode\mskip-.333333\thinmuskip\relax\else\kern-.0555556em\fi}
\def\vsk#1>{\vskip#1\baselineskip}
\def\vv#1>{\vadjust{\vsk#1>}\ignorespaces}
\def\vvn#1>{\vadjust{\nobreak\vsk#1>\nobreak}\ignorespaces}
\def\vvgood{\vadjust{\penalty-500}} \let\alb\allowbreak
 
 \let\sssty\scriptscriptstyle

\def\plait#1{\par\hangindent2\parindent\indent\kern\parindent
\llap{#1\enspace}\ignorespaces}

\def\lsym#1{#1\alb\dots\relax#1\alb} \def\lc{\lsym,}

\let\Smallskip\smallskip
\def\smallskip{\par\Smallskip}
\let\Medskip\medskip
\def\medskip{\par\Medskip}
\let\Bigskip\bigskip
\def\bigskip{\par\Bigskip}

\let\Maketitle\maketitle
\def\maketitle{\Maketitle\thispagestyle{empty}\let\maketitle\empty}

\def\beq{\begin{equation}}
\def\eeq{\end{equation}}
\def\be{\begin{equation*}}
\def\ee{\end{equation*}}

\def\bean{\begin{eqnarray}}
\def\eean{\end{eqnarray}}
\def\bea{\begin{eqnarray*}}
\def\eea{\end{eqnarray*}}
\def\Ref#1{{\rm(\ref{#1})}}

\let\dl\delta  
 \let\eps\varepsilon \let\epsilon\eps

\let\la\lambda

 \let\phi\varphi

\let\Tilde\widetilde

\let\der\partial

\let\geq\geqslant

\let\leq\leqslant

\let\ox\otimes

\def\nin{\not\in}

\def\R{\Bbb R}
\def\C{\Bbb C}
\def\Z{\Bbb Z}

\let\on\operatorname

\def\End{\on{End}}

\def\Im{\on{Im}}
\def\qdet{\on{qdet}}

\def\Wr{\on{Wr}}

\newtheorem{theorem}{Theorem}[section]

\newtheorem{lem}[theorem]{Lemma}
\newtheorem{prop}[theorem]{Proposition}

\theoremstyle{definition}

\theoremstyle{remark}

\numberwithin{equation}{section}

\let\bs\boldsymbol
\let\mc\mathcal
\let\nc\newcommand

\nc{\ch}{\mbox{ch}}
\nc{\pone}{{\mathbb C}{\mathbb P}^1}
\nc{\pa}{\partial}
\nc{\F}{{\euc F}}
\nc{\ri}{\rangle}
\nc{\lef}{\langle}
\nc{\W}{{\euc W}}
\nc{\ep}{\epsilon}
\nc{\su}{\widehat{{\mathfrak sl}}_2}
\nc{\sw}{{\mathfrak s}{\mathfrak l}}
\nc{\g}{{\mathfrak g}}
\nc{\h}{{\mathfrak h}}
\nc{\n}{{\mathfrak n}}
\nc{\N}{\widehat{\n}}
\nc{\G}{\widehat{\g}}
\nc{\De}{\Delta_+}
\nc{\gt}{\widetilde{\g}}
\nc{\Ga}{\Gamma}
\nc{\one}{{\mathbf 1}}
\nc{\z}{{\mathfrak Z}}
\nc{\zz}{{\euc Z}}
\nc{\Hh}{{\euc H}_\beta}
\nc{\qp}{q^{\frac{k}{2}}}
\nc{\qm}{q^{-\frac{k}{2}}}
\nc{\wt}{\widetilde}
\nc{\qn}{\frac{[m]_q^2}{[2m]_q}}
\nc{\cri}{_{\on{cr}}}
\nc{\kk}{h^\vee}
\nc{\sun}{\widehat{\sw}_N}
\nc{\hh}{\widehat{\mathfrak h}}
\nc{\HH}{{\euc H}_{q,t}}
\nc{\ca}{\wt{{\euc A}}_{h,k}(\sw_2)}
\nc{\gl}{\widehat{{\mathfrak g}{\mathfrak l}}_2}
\nc{\el}{\ell}
\nc{\s}{{\mathbf s}}
\nc{\WW}{\W_\beta}
\nc{\scr}{{\mathbf S}}
\nc{\ab}{{\mathbf a}}
\nc{\rr}{r}
\nc{\ol}{\overline}
\nc{\con}{qt^{-1} + q^{-1}t}
\nc{\den}{q^{\el-1} t^{-\el+1}+ q^{-\el+1} t^{\el-1}}
\nc{\ds}{\displaystyle}
\nc{\B}{B}
\nc{\A}{{\mathbb A}}
\nc{\GG}{{\euc G}}
\nc{\UU}{{\euc U}}
\nc{\MM}{{\euc M}}
\nc{\CC}{{\euc C}}
\nc{\GL}{{}^L G}
\nc{\dzz}{\frac{dz}{z}}
\nc{\rep}{{\euc R}ep \;}
\nc{\uqg}{U_q \G}
\nc{\uqgg}{U_q \g}
\nc{\Fq}{{\mathbb F}_q}

\nc{\stimes}{\ltimes}
\nc{\K}{\hat{\euc K}}
\nc{\Ql}{\ol{\mathbb Q}_\ell}

\nc{\ga}{\gamma}
\nc{\PL}{{}^L P}

\nc{\E}{\mc E}
\nc{\mbf}{\mathbf}
\nc{\bb}{{\mathfrak b}}
\nc{\OO}{{\mc O}}
\nc{\Po}{{\mc P}}
\nc{\V}{{\mc V}}
\nc{\yy}{{\mc Y}}
\nc{\M}{\euc M}
\nc{\Coh}{{{\euc C}oh}}
\nc{\Cohn}{\Coh_n}
\nc{\Gaf}{{\mathbb G}_{a,\Fq}}
\nc{\KK}{{\mathfrak k}}

\def\gln{\mathfrak{gl}_N}

\def\Bc{\euc B}
\def\Dc{\euc D}

\def\Uc{\euc U}
\def\Zc{\euc Z}

\def\Rc{\check R}
\def\Qb{\,\overline{\!Q\<\<}\>\vphantom{Q}}
\def\Qbb{\,\overline{\!\bs Q\<\<}\>\vphantom{\bs Q}}
\def\pti{\tilde p}
\def\Qt{\Tilde Q}
\def\Qbt{\Tilde{\bs Q}}
\def\Tt{\Tilde T}
\def\Vt{\Tilde V}
\def\zti{\tilde z}
\def\ztb{\tilde{\bs z}}

\def\half{{\sssty 1\</\<2}}

\def\Real{\on{Re}}
\def\Imag{\on{Im}}

\begin{document}
\title[Reality property of discrete Wronski map with imaginary step]
{Reality property of discrete Wronski map\\[3pt] with imaginary step}

\author[E.\,Mukhin, V\<.\,Tarasov, and \>A.\<\,Varchenko]
{E.\,Mukhin$\>^*$, V\<.\,Tarasov$\>^\star$, and \>A.\<\,Varchenko$\>^\diamond$}

\maketitle

\begin{center}
\vsk-.2>
{\it $\kern-.4em^{*,\star}\<$Department of Mathematical Sciences,
Indiana University\,--\>Purdue University Indianapolis\kern-.4em\\
402 North Blackford St, Indianapolis, IN 46202-3216, USA\/}

\medskip
{\it $^\star\<$St.\,Petersburg Branch of Steklov Mathematical Institute\\
Fontanka 27, St.\,Petersburg, 191023, Russia\/}

\medskip
{\it $^\diamond\<$Department of Mathematics, University of North Carolina
at Chapel Hill\\ Chapel Hill, NC 27599-3250, USA\/}
\end{center}

{\let\thefootnote\relax
\footnotetext{\vsk-.8>\noindent
$^*$\,Supported in part by NSF grant DMS-0900984\\
$^\star$\,Supported in part by NSF grant DMS-0901616\\
$^\diamond$\,Supported in part by NSF grant DMS-0555327}}

\medskip

\begin{abstract}
For a set of quasi-exponentials with real exponents, we consider the discrete
Wronskian (also known as Casorati determinant) with pure imaginary step
\,$2h$. We prove that if the coefficients of the discrete Wronskian are real
and for every its roots the imaginary part is at most \,$|h|$, then the complex
span of this set of quasi-exponentials has a basis consisting of
quasi-exponentials with real coefficients. This result is a generalization of
the statement of the B.~and M.\,Shapiro conjecture on spaces of polynomials.
The proof is based on the Bethe ansatz for the {\sl XXX\/} model.
\end{abstract} \maketitle

\section*{Introduction}
It follows from the Fundamental Theorem of Algebra that a polynomial
$p(x)\in\C[x]$ has real coefficients (up to an overall constant)
if all its roots are real.
The B.~and M.\,Shapiro conjecture for the case of Grassmannian, proved
in~\cite{shapiro}, is a generalization of this fact. It claims that if the
Wronskian of \,$N$ polynomials with complex coefficients has real roots only,
then the space spanned by these polynomials has a basis consisting of
polynomials with real coefficients. For $N=2$ (the case proved in~\cite{EG}),
the statement can be formulated in the following attractive form:
if all critical points of a rational function are real, then the function is
real modulo a fractional linear transformation.

Several generalizations of these results are known, see \cite{EGSV},
\cite{dshapiro}. In \cite{dshapiro}, it is shown that if the Wronskian of $N$
quasi-exponentials \,$p_i(x)\>e^{\la_ix}$, where \,$\la_1\lc\la_N$ are real
numbers and $p_1\lc p_N$ are polynomials with complex coefficients, has real
roots only, then the space spanned by these quasi-exponentials has a basis such
that all polynomials have real coefficients. It is also shown that the result
holds true if the Wronskian is replaced with the discrete Wronskian with real
step and some additional restrictions on the roots of the Wronskian are
imposed. A similar statement is proved in \cite{dshapiro} about spaces of
quasi-polynomials of the form $x^{z_i}p_i(x,\log x)$, where $z_i$ are real
numbers and $p_i(x,y)$ are polynomials with complex coefficients, and their
Wronskians.

In this paper we present yet another instance of this phenomenon. We consider
the discrete Wronskian with purely imaginary step $2h$ of a space of $N$
quasi-exponentials $e^{\la_ix}p_i(x)$ with real exponents $\la_i$. We show
that if this discrete Wronskian is real and for every its roots the imaginary
part is at most \,$|h|$, then the the space of quasi-exponentials has a basis
such that all polynomials have real coefficients.

Our method is similar to that of \cite{shapiro}, \cite{dshapiro}. Namely,
we consider the quantum integrable model of {\sl XXX\/}-type associated with
a tensor product of vector representations of $\gln$. This model can be solved
by the method of the algebraic Bethe ansatz. For each space of
quasi-exponentials $V$ of dimension $N$ in a generic position, we associate
an eigenvector of all transfer matrices of the model (Bethe vector), such that
the eigenvalues of the transfer matrices are the coefficients of the difference
operator of order \,$N$ annihilating $V$. Then we use properties of the
transfer matrices to show an appropriate symmetry of the coefficients, which
implies our result.

The paper is organized as follows. In Section \ref{form sec} we describe
our result, see Theorem \ref{theorem1}, and give some examples.
Section~\ref{sec2} provides the required results from the representation
theory of the Yangian $Y(\gln)$. Section~\ref{proof sec} contains the proof of
the main theorem. We discuss curious reformulations of Theorem~\ref{theorem1}
in Section~\ref{ref sec}. The Appendix contains a proof of the symmetry of
transfer matrices, see Proposition~\ref{symmetry}, used in this paper.

\section{Formulation of the results}\label{form sec}
\subsection{The main theorem}
A function of the form \,$p(x)\>Q^{x}$, where \>$Q$ is a nonzero complex
number with the argument fixed, and $p(x)\in\C[x]$ is a polynomial,
is called {\it a quasi-exponential function with base \,$Q$}.
A quasi-exponential function \,$p(x)\>Q^{x}$ is called {\it real}
if $Q$ is real and $p(x)$ has real coefficients, $p(x)\in\R[x]$.

Fix a natural number $N\geq 2$. Let $\bs Q=(Q_1,\dots,Q_N)$ be a sequence
of nonzero complex numbers with their arguments fixed.
We call a complex vector space of dimension $N$ spanned by quasi-exponential
functions \,$p_i(x)\>Q_i^x$, with $i=1,\dots,N$, a {\it space of
quasi-exponentials with bases $\bs Q$}. The space of quasi-exponentials is
called {\it real\/} if it has a basis consisting of real quasi-exponential
functions.

Let $h$ be a non-zero complex number.
{\it The discrete Wronskian with step $2h$},
also known as Casorati determinant,
of functions $f_1(x),\dots,f_N(x)$ is the $N\times N$ determinant
\be
\label{dwr}
\Wr^d_h(f_1,\dots, f_N)\,=\,\det\>\Bigl(f_i\bigl(x+h(2j-N-1)\bigr)\Bigr)
_{i,j=1\lc N}\,.
\ee
Note that the notation here is slightly different from that of \cite{dshapiro}.

The discrete Wronskians of two bases for a space of functions
differ by multiplication by a nonzero number.

Let $V$ be a space of quasi-exponentials with bases $\bs Q$.
The discrete Wronskian of any basis for $V$ is a quasi-exponential function of the form
$w(x)\prod_{j=1}^N Q_j^x$, where $w(x)\in\C[x]$.
The unique representative with a monic polynomial
$w(x)$ is called {\it the discrete Wronskian of\/} $V$ and is denoted by
$\Wr^d_h(V)$.

The main result of this paper is the following theorem.

\begin{theorem}
\label{theorem1}
Let $V$ be a space of quasi-exponentials with real bases
\vvn.1>
\,$\bs Q\in(\R^\times)^N$,
\,and let $\Wr^d_h(V)=w(x)\prod_{j=1}^NQ_j^x$, where
\vvn.16>
$w(x)=\prod_{i=1}^n(x-z_i)$. Assume that \;$\Real h=0$, \;$w(x)\in\R[x]$, and
\;$|\<\Imag z_i|\leq |h|$ \,for all $i=1\lc n$. \,Then the space $V$ is real.
\end{theorem}

Theorem \ref{theorem1} is proved in Section \ref{proof sec}.

\medskip
The B.~and M.\,Shapiro conjecture proved in \cite{shapiro} asserts that if all
roots of the differential Wronski determinant of a space of polynomials $V$
are real then $V$ is real. The B.~and M.\,Shapiro conjecture follows from
Theorem~\ref{theorem1} by taking the limit $h\to 0$.

\subsection{Examples}
\strut
\medskip
\noindent{\it Example 1.}
Let $N=2$, \,$\bs Q=(1,Q)$, \,$p_1(x)=x+a$, \,$p_2(x)=x+b$.
Then the discrete Wronskian has two roots. Let the discrete Wronskian be a real
function. Then without loss of generality we can assume its zeros to be at
$\pm A$, where $A$ is either real or pure imaginary. We have the following
equation on $a,b$,
\vvn.3>
\be
\Wr^d_h(x+a,Q^x(x+b))=Q^x(Q^h-Q^{-h})(x+A)(x-A)\,,
\vv.3>
\ee
which has two solutions:
\vvn-.2>
\be
a\,=\,-b\,=\,\frac{(Q^{h}+Q^{-h})h\pm\sqrt{(Q^{h}-Q^{-h})^2A^2+4h^2}}
{(Q^{h}-Q^{-h})}\;.
\vv.4>
\ee
If \;$\Real h=0$ and \,$Q\in\R$, then $|Q^h|=1$, \,$Q^h+Q^{-h}$ is real, and
\,$Q^h-Q^{-h}$ is purely imaginary. Hence, $a,b$ are real for real $A$, and
for purely imaginary $A$ such that $|A|\leq |h|$.

However, if $A$ is purely imaginary and $|A|>|h|$, then for $Q=e^{\pi/2h}$
the numbers $a,b$ are not real. In this example, Theorem~\ref{theorem1} gives
the largest possible set of values of $A$ such that the numbers $a$ and $b$ are
real for all real \,$Q\ne0$.

\medskip
\noindent{\it Example 2.}
Let $N=2$, \,$\bs Q=(1,1)$, \,$p_1(x)=x+a$, $p_2(x)=x^3+bx^2+c$. Then the
discrete Wronskian has three roots, which we assume to be at $0$, $A$ and $B$.
We also assume that both $A$ and $B$ are real, or $A$ and $B$ are complex
conjugate.

This case corresponds to the following equation on $a,b,c$,
\vvn.3>
\be
\Wr^d_h(x+a,x^3+bx^2+c)=4hx(x-A)(x-B)\,,
\vv.1>
\ee
which has two solutions:
\vvn.3>
\begin{align*}
a\, &{}=\, -(A+B)/3\pm 1/3\sqrt{-AB-3h^2+A^2+B^2},\\
b\, &{}=\, -(A+B)\mp \sqrt{-AB-3h^2+A^2+B^2}, \\
c\, &{}=\, (-4/3+2h^2)(A+B)+h^2/3\sqrt{-AB-3h^2+A^2+B^2}\,.\\[-14pt]
\end{align*}
Let \;$\Real h=0$. Then $a,b,c$ are real for real $A,B$. If $A$ and $B$ are
complex conjugate, then $a,b,c$ are real if and only if
\;$3\<\>(\Imag A)^2-(\Real A)^2\leq 3|h|^2$.

The equation \;$3\<\>(\Imag A)^2-(\Real A)^2=3|h|^2$ defines a hyperbola
on the complex plane which is tangent to the lines $\Im A=\pm|h|$.

\section{Transfer matrices and the Bethe subalgebra}
\label{sec2}
In this section we recall the required results from the representation
theory of the Yangian \,$Y(\gln)$.

\medskip
Let $W=\C^N$ with a chosen basis $v_1,\dots,v_N$.
For an operator $M\in\End(W)$, we denote
$M_{(i)}=1^{\ox(i-1)}\ox M\ox 1^{\ox(n-i)}$.
Similarly, for an operator
$M\in\on{End}(W^{\ox 2})$, we denote by $M_{(ij)}\in\End(W^{\ox n})$
the operator acting as $M$ on the $i$-th and $j$-th factors of $W^{\ox n}$.

Let $E_{ab}\in\End(W)$ be the linear operator with the matrix
$(\delta_{ia}\delta_{jb})_{i,j=1\lc N}$.
Let $R(x)$ be the {\it rational $R$-matrix\/},
\vvn-.4>
\be
R(x)\,=\,1+x^{-1}\sum_{a,b=1}^N\,E_{ab}\ox E_{ba}\,=\,1+x^{-1}\>P\,.
\vv-.1>
\ee
where \,$P\in\End(W^{\ox 2})$ \,is the flip map:
\,$P(v\ox w)=w\ox v$ \,for all \,$v,w\in W$.

The {\it Yangian\/} $Y(\gln)$ is the unital associative algebra over $\C$
with generators $T_{ab}^{\{s\}}$, \,$a,b=1\lc N$, \,$s\in\Z_{\geq 1}$, and
relations
\beq
\label{rel}
R_{(12)}(x-y)T_{(1)}(x)T_{(2)}(y)=T_{(2)}(y)T_{(1)}(x)R_{(12)}(x-y)\,,
\eeq
where \;$T(x)=\sum_{a,b=1}^N E_{ab}\ox T_{ab}(x)$ \;and
\;$T_{ab}(x)=\delta_{ab}+\sum_{s=1}^\infty T_{ab}^{\{s\}}x^{-s}$.
\vvn.16>
The Yangian $Y(\gln)$ is a Hopf algebra, with the coproduct and antipode
given by
\vvn.2>
\begin{gather}
\Delta\bigl(T_{ab}(x)\bigr)\,=\,\sum_{i=1}^N\,T_{ib}(x)\ox T_{ai}(x)\,,
\notag
\\
\label{antipode}
\sum_{a,b=1}^N E_{ab}\ox S\bigl(T_{ab}(x)\bigr)\,=\,\bigl(T(x)\bigr)^{-1}\,.
\end{gather}
The Yangian $Y(\gln)$ is a flat deformation of $U\bigl(\gln[t]\bigr)$,
the universal enveloping algebra of the current Lie algebra $\gln[t]$.

Given $z\in \C$, define the $Y(\gln)$-module structure on the space $W$
by letting \,$T_{ab}(x)$ act as $\delta_{ab}+(x-z)^{-1}E_{ba}$.
We denote this module $W(z)$ and call it the {\it evaluation module}.

\medskip
Let \,$\bs Q=(Q_1,\dots,Q_N)\in(\C^\times)^N$, and
\,$Q=\on{diag}(Q_1,\dots,Q_N)$ be the diagonal matrix with diagonal entries
$Q_a$. \,Let $\der=\der/\der x$.
\,Set \,$X_{ab}=\dl_{ab}\<-Q_a\>T_{ab}(x)\>e^{-\der}$.
Define the universal difference operator by
\vvn-.2>
\be
\Dc_{\bs Q}\,=\,\sum_{\sigma\in S_N}\,(-1)^\sigma\,
X_{1\sigma(1)}\>X_{2\sigma(2)}\dots X_{N\sigma(N)}\,.
\vv-.4>
\ee
Write
\vvn-.1>
\be
\Dc_{\bs Q}\,=\,1-B_{1,\bs Q}(x)e^{-\der}+B_{2,\bs Q}(x)e^{-2\der}-
\dots +(-1)^N B_{N,\bs Q}(x)e^{-N\der}\,,
\vv.3>
\ee
where \,$B_{i,\bs Q}(x)$ are series in $x^{-1}$ with coefficients in $Y(\gln)$.
The series $B_{i,\bs Q}(x)$ coincides with the higher transfer matrices,
introduced in~\cite{KS}, see~\cite{CT}, \cite{bethe}, and formula~\Ref{B=T}.
Set \,$B_{0,\bs Q}(x)=1$.

\smallskip
The unital subalgebra of $Y(\gln)$ generated by the coefficients of the series
$B_{i,\bs Q}(x)$, $i=1,\dots,N$, is called the {\it Bethe algebra} and denoted
by $\Bc_{\bs Q}$. The Bethe algebra is commutative \cite{KS}.

\smallskip
Given \,$\bs z=(z_1,\dots,z_n)\in\C^n$, \,denote by
\,$\bs W({\bs z})=W(z_1)\ox\dots\ox W( z_n)$ \, the tensor product of
evaluation modules.
For \,$i=0\lc N$, \,let \,$B_{i,\bs Q}(x;\bs z)$ be the image of the series
$B_{i,\bs Q}(x)$ in $\End\bigl(\bs W(\bs z)\bigr)[[x^{-1}]]$. The series
$B_{i,\bs Q}(x;\bs z)$ sums up to a rational function in $x$. We have
\vvn.2>
\be
B_{N,\bs Q}(x;\bs z)\,=\,b_{\bs Q}(x;\bs z)\cdot\on{Id}\,,\qquad
b_{\bs Q}(x;\bs z)\,=\,Q_1\dots Q_N\,\prod_{i=1}^n\,\frac{x-z_i+1}{x-z_i}\;.
\vvgood
\vv.2>
\ee

Let \,$v\in\bs W(\bs z)$ be an eigenvector of the Bethe algebra $\Bc_{\bs Q}$,
\vvn.3>
\be
B_{i,\bs Q}(x;\bs z)\,v\,=\,B_{i,\bs Q,v}(x;\bs z)\,v\,,\qquad i=1\lc N\,,
\vv.3>
\ee
where \,$B_{i,\bs Q, v}(x;\bs z)$ are rational functions in $x$ with complex
coefficients. Let \,$\Dc_{\bs Q,v}(x;\bs z)$ be the scalar difference operator
\vvn.3>
\be
\Dc_{\bs Q, v}(x;\bs z)\,=\,1-B_{1,\bs Q,v}(x;\bs z)e^{-\der}+
B_{2,\bs Q,v}(x;\bs z)e^{-2\der}-\dots +
(-1)^N B_{N,\bs Q, v}(x;\bs z)e^{-N\der}\,.
\vv.3>
\ee

Let $\Dc$ be a scalar difference operator. We call the space
\vvn.3>
\be
\{\>f(x)\ |\ \Dc f(x)=0, \ \,f(x)
\ \text{is a linear combination of quasi-exponential functions}\>\}
\vv.2>
\ee
the {\it quasi-exponential kernel of the operator\/} \,$\Dc$\,.

\smallskip
Let \,$\Uc$ \,be the complex span of \,$1$-periodic quasi-exponentials
\,$e^{2\pi\sqrt{-1}\,kx}$, \,$k\in\Z$. 

\begin{prop}[\cite{generating}]
\label{cite2}
Let \,$\bs z$ \,and \,$\bs Q$ be generic. \,For every $N$-dimensional
complex space $V$ of quasi-exponentials with bases \,$\bs Q$ such that
\be
\Wr^d_\half(V)\,=\,
\prod_{i=1}^n\,\bigl(x-z_i+(N+1)/2\bigr)\;\prod_{j=1}^N Q_j^x\,,
\ee
there exists
an eigenvector \,$v\in \bs W(z)$ of the Bethe algebra \,$\Bc_{\bs Q}$ such that
the quasi-exponential kernel of the operator \,$\Dc_{\bs Q,v}(x;\bs z)$ has
the form \,$V\ox\Uc$\,.
\qed
\end{prop}

Recall that the module $\bs W(\bs z)$ as a vector space is $W^{\ox n}$.
Let \,$\langle\cdot\,,\<\cdot\rangle$ be the standard sesquilinear form
on $W^{\ox n}$:
\vvn-.6>
\be
\langle v_{a_1}\ox\dots\ox v_{a_n},
v_{b_1}\ox\dots\ox v_{b_n}\rangle=\prod_{i=1}^n\delta_{a_ib_i}\,.
\vv-.1>
\ee
We assume that the form is linear with respect to the first argument and
semilinear with respect to the second one.
The form is clearly a positive definite form.

\smallskip
The following proposition is the key technical fact about
the Bethe algebra we need for the proof of Theorem~\ref{theorem1}.

\begin{prop}\label{symmetry} We have
\begin{enumerate}
\item For every \,$i=1,\dots,n-1$, \,and \;$j=1,\dots,N$,
\beq
\label{1}
\Rc_{(i,i+1)}(z_i-z_{i+1})\>B_{j,\bs Q}(x;\bs z)\,=\,
B_{j,\bs Q}(x;\sigma_{i,i+1}\>\bs z)\>\Rc_{(i,i+1)}(z_i-z_{i+1})\,,
\eeq
where \;$\Rc_{(i,i+1)} (x)=xP_{(i,i+1)}+1$ \;and
\;$\sigma_{i,i+1}\>\bs z=(z_1,\dots,z_{i-1},z_{i+1},z_i,z_{i+2},\dots,z_n)$\,.

\smallskip
\item For every $j=0,\dots, N$, \,and \,$v,w\in \bs W(\bs z)$,
\beq
\label{2}
\bigl\langle B_{j,\bs Q}(x;\bs z)\>v,w\bigr\rangle\,=\,b_{\bs Q}(x;\bs z)\,
\bigl\langle v, B_{N-j,\Qbb^{-1}}(-\bar x-1;-\bar{\bs z})\>w\bigr\rangle\,,
\eeq
where \;$-\bar{\bs z}=(-\bar z_1,\dots,-\bar z_n)$,
\;$\Qbb^{-1}=(\Qb_1^{-1},\dots,\Qb_N^{-1})$\,,
\;and the bar denotes the complex conjugation
\end{enumerate}
\end{prop}
Proposition \ref{symmetry} is proved in Appendix.

\section{Proof of Theorem~\ref{theorem1}}
\label{proof sec}

Recall that $Q_1,\dots,Q_N$ are nonzero real numbers, $h\ne 0$ is a purely
imaginary number and $V$ is a space of quasi-exponentials $V$ with bases
$\bs Q$. We also assume that
\be
\Wr^d_h(V)\,\prod_{j=1}^N\,Q_j^{-x}\>=\,\prod_{i=1}^n\,(x-z_i)
\ee
is a polynomial with real coefficients, and \;$|\Imag z_i\>|\leq|h|$ \,for all
\,$i=1\lc n$\,.

\begin{lem}
\label{31}
Assume that the conclusion of Theorem~\ref{theorem1} holds for generic
$\bs Q,\bs z=(z_1\lc z_n)$ satisfying the assumption of
Theorem~\ref{theorem1}. Then the conclusion of Theorem~\ref{theorem1} holds for
arbitrary $\bs Q,\bs z$ satisfying the assumption of Theorem~\ref{theorem1}.
\end{lem}

\begin{proof}
The proof follows the reasoning in Sections~3.1, 3.2 of~\cite{dshapiro}.
\end{proof}

To use the results from Section~\ref{sec2}, we make a change of variables. Set
\vvn.3>
\be
\Qt_i=Q_i^{2h}\>,\qquad \zti_i=\>\frac{z_i}{2h}\,,\qquad i=1\lc N\,,
\ee
and let $\Vt$ be the space of quasi-exponentials with bases $\Qbt\in\C^N$,
given by
\vvn.3>
\beq
\label{Vt}
\Vt\,=\,\bigl\{\>p\bigl(2hx+h\>(N+1)\bigr)\>\Qt^x \ \>|
\ \,p(x)\>Q^x\in V\>\}\,.
\eeq
Then
\vvn-.4>
\be
\Wr^d_\half(\Vt)\,=\,
\prod_{i=1}^n\,\bigl(x-\zti_i+(N+1)/2\bigr)\,\prod_{j=1}^N\,\Qt_j^x\,.
\ee

\smallskip
Since the polynomial \,$\prod_{i=1}^n(x-z_i)$ \,has real coefficients,
there is $k$, where $0\leq 2k\leq n$, and an enumeration of $z_1\lc z_n$
such that that \,$\bar z_{2i-1}=z_{2i}$, for \,$i=1\lc k$, \,and
\,$z_{2k+1},\dots,z_n$ are real. Define the form
\,$\langle\cdot\,,\<\cdot\rangle_k$ on $W^{\ox n}$ by the formula
\be
\langle v,w\rangle_k\>=\,\Bigl\langle v\>,\,
\prod_{i=1}^k\,\Rc_{(2i-1,2i)}(\zti_{2i-1}-\zti_{2i})\>w\Bigr\rangle\,.
\ee

\begin{lem}
Assume \;${|\Imag z_i\>|<|h|}$ \,for all \,$i=1\lc 2k$\,. \,Then
the form \,$\langle\cdot\,,\<\cdot\rangle_k$ \,is positive definite.
\end{lem}
\begin{proof}
The linear operators \,$\Rc_{(2i-1,2i)}(\zti_{2i-1}-\zti_{2i})$ \,with
$i=1,\dots,k$ are pairwise commuting. Moreover, under our assumptions
each \,$\Rc_{(2i-1,2i)}(\zti_{2i-1}-\zti_{2i})$ \,is a positive definite
selfadjoint operator with respect to the standard form
\,$\langle\cdot\,,\<\cdot\rangle\,$. The lemma follows.
\end{proof}

By Lemma~\ref{31} we can assume that \,$\bs Q$ and \,$\bs z$ are generic.
Then by Proposition~\ref{cite2}, there exists an eigenvector $v\in\bs W(\ztb)$
of the Bethe algebra $\Bc_{\Qbt}$ such that $\Vt$ is in the kernel of
the scalar difference operator $\Dc_{\Qbt,v}(x;\ztb)$.

\smallskip
By Proposition \ref{symmetry}, we have
\vvn.1>
\be
(B_{j,\Qbt}(x;\ztb))^*\,=\,
\overline{b_{\Qbt}(x;\ztb)} B_{N-j,\Qbt}(-\bar x-1;\ztb)\,,
\vv.2>
\vvgood
\ee
where the asterisk denotes the Hermitian conjugation with respect
to the form \,$\langle\cdot\,,\<\cdot\rangle_k$\>. Since the form
\,$\langle\cdot\,,\<\cdot\rangle_k$ is positive definite, it yields
\vvn.3>
\beq
\label{Bbar}
\overline{B_{j,\Qbt,v}(x;\ztb)}\,=\,
\overline{b_{\Qbt}(x;\ztb)}\,B_{N-j,\Qbt,v}(-\bar x-1;\ztb)\,.
\vv.1>
\eeq

Let \,$p(x)\>Q_i^x\in V$, where \,$p(x)$ \,is monic.
\,Set \,$\pti(x)=p\bigl(2hx+h\>(N+1)\bigr)$\,, \,see~\Ref{Vt}.
\,Then \,$\Dc_{\Qbt,v}(x;\ztb)\>\bigl(\pti(x)\Qt_i^x\bigr)=0$\,,
and by formula~\Ref{Bbar},
\vvn.2>
\be
\Dc_{\Qbt,v}(x;\ztb)\>\bigl(\,\overline{\pti(-\bar x-N-1)}\>\Qt_i^x\bigr)
\,=\,0\,.
\vv.2>
\ee
For generic \,$\Qbt\>$, the equation
\,$\Dc_{\Qbt,v}(x;\ztb)\>\bigl(q(x)\>\Qt_i^x\bigr)=0$ \,determines a polynomial
\,$q(x)$ \,unique\-ly up to multiplication by a constant. Hence
\,$\overline{\pti(-\bar x-N-1)}=\pti(x)$, \,because the top coef\-ficients of
both polynomials are the same. Taking \,$y=2hx+h\>(N+1)$ \,gives
\,$\overline{p(\bar y)}=p(y)$. Theorem~\ref{theorem1} is proved.

\section{Matrix reformulation}
\label{ref sec}

In this section we give a matrix formulation of Theorem \ref{theorem1}.

\smallskip
Let \,$\Zc$ be an $N\times N$ matrix with entries \;$\Zc_{ii}=\>a_i$ \,for
\,$i=1\lc N$, \,and
\vvn.3>
\be
\Zc_{ij}\>=\,\frac1{\sin(\la_i-\la_j)}\,,\qquad 1\leq i<j\leq N\,.
\vv.2>
\ee
Here \,$a_1\lc a_N$ \,and \,$\la_1\lc\la_N$ are complex numbers.
For \,$i=1\lc n$, \,set
\vvn.3>
\be
p_i(x)\,=\,x-a_i-\sum_{j\ne i}\cot(\la_i-\la_j)\,.
\vv-.1>
\ee
Let $V$ be the space
of quasi-exponentials spanned by \,$p_i(x)\>e^{\la_ix}$, \;$i=1\lc N$.

\begin{lem}
\label{Wron}
We have \;$\det(x-\Zc)\,=\,\Wr^d_{\sqrt{-1}}(V)\,e^{-\!\sum_{i=1}^N\la_ix}$\,.
\end{lem}
\begin{proof}
The statement can be obtained by the same calculation as in Section~6.1
of~\cite{dshapiro}.
\end{proof}

\begin{theorem}
\label{theorem1a}
Let \,$\la_1,\dots,\la_N$ be real numbers such that \,$\la_i-\la_j\nin\pi\Z$
\,for all \,$i,j=1\lc N$. Assume that the characteristic
polynomial \,$\det(x-\Zc)=\prod_{i=1}^N(x-z_i)$ \,has real coefficients and
\;$|\<\Imag z_i|\leq 1$ \,for all $i=1\lc N$.
Then the numbers \,$a_1,\dots,a_N$ are real.
\end{theorem}
\begin{proof}
The statement follows from Theorem~\ref{theorem1} and Lemma~\ref{Wron}.
\end{proof}

If \,$a_i=b_i/\eps$\,, \;$\la_i=\eps\mu_i$\,, \;$z_i=\eps x_i$ \,for all
\,$i=1\lc N$, and \,$\eps\to 0$, then Theorem~\ref{theorem1a} turns into
Theorem~6.4 of \cite{dshapiro}.

\appendix
\section*{Appendix}
\refstepcounter{section}

Given \,$\bs i=(i_1<\dots<i_k)$ \,and \,$\bs j=(j_1<\dots<j_k)$, \,set
\vvn.3>
\beq
\label{Tk}
T^{\wedge k}_{\bs{ij}}(x)\,=\,\sum_{\sigma\in S_k}\,(-1)^\sigma\,
T_{i_1\>j_{\sigma(1)}}(x)\>T_{i_2\>j_{\sigma(2)}}(x-1)\dots
T_{i_k\>j_{\sigma(k)}}(x-k+1)\,.
\eeq
Denote \;$\qdet(x)=T^{\wedge N}_{\bs{ii}}(x)$, \,where \,$\bs i=(1\lc N)$.
We have
\vvn.3>
\beq
\label{B=T}
B_{k,\bs Q}(x)\,=\!\sum_{\bs i\>=(i_1<\dots<i_k)}\!\!
Q_{i_1}\<\dots Q_{i_k}\,T^{\wedge k}_{\bs{ii}}(x)\,.
\eeq

For every \,$\bs i=(i_1<\dots<i_k)$, \,set \,$\bs i'\<=(i'_1<\dots<i'_{N-k})$,
\,where
\vvn.2>
\be
\{i_1\lc i_k,i'_1\lc i'_{N-k}\}=\{1,2,\dots,N\}\,.
\vv.3>
\ee
The following result is known, see~(1.11) in~\cite{NT}.

\begin{theorem}\cite{NT}
We have
\vvn.1>
\beq
\label{NT}
S\bigl(T^{\wedge k}_{\bs{ij}}(x)\bigr)\,\qdet(x)\,=\,
T^{\wedge (N-k)}_{\bs j'\<\bs i'}(x-k)\,,
\vv.2>
\eeq
where $S$ is the antipode.
\end{theorem}

For the $Y(\gln)$-module \,$W(\bs z)$, \,let
\,$\rho_{\bs z}:Y(\gln)\to\End(W^{\ox N})$
\,be the corresponding al\-gebra homomorphism. Denote
\vvn.2>
\be
T(x;\bs z)\,=\,\sum_{a,b=1}^N E_{ab}\ox\rho_{\bs z}\bigl(T_{ab}(x)\bigr)\,,
\qquad
\Tt(x,\bs z)\,=\,\sum_{a,b=1}^N E_{ab}\ox\rho_{\bs z}(S\bigl(T_{ab}(x))\bigr)
\,.
\vv-.3>
\ee
We have
\beq
\label{Txz}
T(x;\bs z)\,=\,R_{(0N)}(x-z_N)\dots R_{(02)}(x-z_2)R_{(01)}(x-z_1)\,,
\vv-.3>
\eeq
and
\be
\Tt(x,\bs z)\,=\,
(R_{(01)}(x-z_1))^{-1}(R_{(02)}(x-z_2))^{-1}\dots(R_{(0N)}(x-z_N))^{-1}\,,
\vv.3>
\ee
see \Ref{antipode}. Here we enumerate the factors of \,$W^{\ox(N+1)}$ by
\,$0,1\lc N$.
Since \,$\bigl(R(x)\bigr)^t=R(x)$ \,and \,$R(x)\>R(-x)=1-x^{-2}$, we obtain
\beq
\label{ST}
\bigl(\Tt(x;\bs z)\bigr)^t\,=\,
T(-x;-\bs z)\,\frac{b(x-1;\bs z)}{b(x;\bs z)}\;,\qquad
b(x;\bs z)\,=\,\prod_{i=1}^N\,\frac{x-z_i+1}{x-z_i}\;.
\eeq

\smallskip
Let \,$T^{\wedge k}_{\bs{ij}}(x;\bs z)=
\rho_{\bs z}\bigl(T^{\wedge k}_{\bs ij}(x)\bigr)$\,.
Then by~\Ref{Tk} and~\Ref{ST},
\be
S\bigl(T^{\wedge k}_{\bs{ij}}(x)\bigr)\,=\,
T^{\wedge k}_{\bs{ji}}(k-1-x;-\bs z)\;\frac{b(x-k;\bs z)}{b(x;\bs z)}\;.
\ee
In addition, it is known that
\;$\rho_{\bs z}\bigl(\qdet(x)\bigr)=b(x;\bs z)$\,.
Hence taking into account~\Ref{NT} and~\Ref{B=T},
after simple transformations we get
\vvn.3>
\be
\label{STk}
T^{\wedge k}_{\bs{ij}}(x)\,=\,b(x;\bs z)\,
T^{\wedge(N-k)}_{\bs i'\<\bs j'}(-x-1;-\bs z)
\vv-.4>
\ee
and
\vv-.2>
\be
B_{k,\bs Q}(x)\,=\,b(x;\bs z)\,B_{N-k,\bs Q^{-1}}(-x-1;-{\bs z})\,.
\vv.4>
\vvgood
\ee
It is also known that
\vvn.3>
\be
\bigl\langle B_{j,\bs Q}(x;\bs z)\>v,w\bigr\rangle\,=\,
\bigl\langle v, B_{j,\Qbb}(\bar x;\bar{\bs z})\>w\bigr\rangle\,,
\vv.4>
\ee
for every $j=1\lc N$, \,and \,$v,w\in \bs W(\bs z)$, see for example,
Proposition~4.11 in~\cite{bethe}. This proves formula~\Ref{2} in
Proposition~\ref{symmetry}.

\smallskip
Formula~\Ref{1} follows from \Ref{B=T}, \Ref{Txz}, and the Yang-Baxter equation
\vvn.3>
\begin{align*}
\Rc_{(i,i+1)}(z_i-z_{i+1})\> &R_{(0,i+1)}(x-z_{i+1})\>R_{(0i)}(x-z_i)\,=\,
\\[3pt]
{}=\,{}& R_{(0,i+1)}(x-z_i)\>R_{(0i)}(x-z_{i+1})\>\Rc_{(i,i+1)}(z_i-z_{i+1})\,.
\end{align*}

\end{document}